\documentclass[11pt]{amsart}
\usepackage{amsfonts,amssymb,amscd,amsmath,enumerate,thmtools,thm-restate,tikz-cd,url,verbatim}
\usepackage[all]{xy}
\SelectTips{eu}{}
\xyoption{curve}
\usepackage[top=1.1in, bottom=1.1in, left=1.4in, right=1.4in]{geometry}
\CompileMatrices

\newcommand{\lra}{\longrightarrow}

\newcommand{\ov}{\overline}

\newcommand{\col}{\colon}
\newcommand{\dd}{\partial}
\newcommand{\fk}{{\mathsf k}}
\newcommand{\fm}{{\mathfrak m}}
\newcommand{\ann}{\operatorname{ann}}

\newcommand{\bd}{\boldsymbol}

\newcommand{\Ima}{\operatorname{Im}}
\newcommand{\gr}{\operatorname{gr}}

\newcommand{\rank}{\operatorname{rank}}

\newcommand{\Tor}{\operatorname{Tor}}
\newcommand{\Ext}{\operatorname{Ext}}

\newcommand{\Hom}{\operatorname{Hom}}
\newcommand{\Po}{\operatorname{P}}
\newcommand{\Eo}{\operatorname{E}}
\newcommand{\To}{\operatorname{T}}

\theoremstyle{plain}
\newtheorem{theorem}{Theorem}[section]

\newtheorem*{Theorem}{Theorem}

\newtheorem*{Main Theorem}{Main Theorem}

\newtheorem{proposition}[theorem]{Proposition}
\newtheorem{lemma}[theorem]{Lemma}
\newtheorem{corollary}[theorem]{Corollary}

{\Alph{theorem}}
{\Alph{theorem}}

\theoremstyle{definition}

\newtheorem{chunk}[theorem]{}

\newtheorem{remark}[theorem]{Remark}

\theoremstyle{remark}

\newenvironment{bfchunk}{\begin{chunk}\textit}{\end{chunk}}

\numberwithin{equation}{theorem}

\newtheorem*{Case1}{Case 1}
\newtheorem*{Case2}{Case 2}

\begin{document}
\title[Cohomology over short Gorenstein rings]{Cohomology of  finite modules \\over short Gorenstein rings}

\author[M.~C.~Menning]{Melissa C.~Menning}
\address{Melissa~Menning\\ Department of Mathematics and Statistics\\
   University of Missouri\\ \linebreak Kansas City\\ MO 64110\\ U.S.A.}
     \email{mcm146@mail.umkc.edu}

\author[L.~M.~\c{S}ega]{Liana M.~\c{S}ega}
\address{Liana M.~\c{S}ega\\ Department of Mathematics and Statistics\\
   University of Missouri\\ \linebreak Kansas City\\ MO 64110\\ U.S.A.}
     \email{segal@umkc.edu}

\subjclass[2000]{Primary 13D07. Secondary 13D02, 13H10}
\keywords{}
\thanks{This work was supported by a grant from the Simons Foundation (\# 354594, Liana \c Sega)  and a research grant from the UMKC School of Graduate Studies (Melissa Menning)}

\begin{abstract}
Let $R$ be a Gorenstein  local ring with maximal ideal $\fm$ satisfying $\fm^3=0\ne\fm^2$. Set $\fk=R/\fm$ and $e=\rank_{\fk}(\fm/\fm^2)$.  If  $e>2$ and $M$, $N$ are finitely generated $R$-modules, we show that the formal power series $$\sum_{i=0}^\infty\rank_{\fk}\left(\Ext^i_R(M,N)\otimes_R\fk \right)t^i\quad\text{and}\quad  \sum_{i=0}^\infty\rank_{\fk}\left(\Tor_i^R(M,N)\otimes_R \fk \right)t^i$$
are rational, with denominator $1-et+t^2$. 
\end{abstract}

\maketitle

\section*{Introduction}
Let $(R,\fm,\fk)$ be a Noetherian commutative local ring;  $\fm$ denotes the maximal ideal and $\fk=R/\fm$. If $L$ is an $R$-module, we set $\nu(L)=\rank_{\fk}(L/\fm L)$. 
 Let $M$ and $N$ be finite (meaning finitely generated) $R$-modules.

We consider the formal power series 
$$
\Eo^{M,N}_R(t)=\sum_{i=0}^\infty\nu\left(\Ext^i_R(M,N)\right)t^i\quad\text{and}\quad \To^R_{M,N}(t)=\sum_{i=0}^\infty\nu\left(\Tor_i^R(M,N)\right)t^i\,.
$$
Note that $\Eo^{M,\fk}_R(t)=\To^R_{M,\fk}(t)=\To^R_{\fk,M}(t)$; this series  is usually called the {\it Poincar\'e series} of $M$, denoted $\Po_M^R(t)$. The series  $\Eo^{\fk,N}_R(t)$ is called the {\it Bass series} of $N$. 

Although rings with transcendental Poincar\'e series  exist, significant classes of rings $R$ are known to satisfy the property that the Poincar\'e series of all finite $R$-modules are rational, sharing a common denominator; see for example \cite{RossiSega} for a recent development.  If this property holds, then the Bass series of all finite $R$-modules are also rational, sharing a common denominator; see \cite[Lemma 1.2]{S1}. 

Less  is known about the series $\Eo^{M,N}_R(t)$ and $\To_{M,N}^R(t)$ for arbitrary $M$, $N$. If $\fm^2=0$, then it is an easy exercise to show that $\left(1-et\right)\cdot \Eo^{M,N}_R(t)\in \mathbb Z[t]$,  where $e=\nu(\fm)$.  When $R$ is a complete intersection of codimension $c$, Avramov and Buchweitz \cite[Proposition 1.3]{AB} showed that $(1-t^2)^c\cdot \Eo^{M,N}_R(t)\in \mathbb Z[t]$  for all finite $M$, $N$. 

We consider $R$ to be Gorenstein, with $\fm^3=0$.   In this case, Sj\"odin \cite{Sj} shows that Poincar\'e series of all finite $R$-modules are rational, sharing a common denominator.  We prove that Sj\"odin's result can be extended as follows: 

\begin{Theorem}
Let $(R,\fm,\fk)$ be a local  Gorenstein ring with $\fm^3=0\ne \fm^2$ and set $e=\nu(\fm)$. If $e>2$ and $M$, $N$ are finite $R$-modules, then  
$$(1-et+t^2)\cdot \Eo^{M,N}_R(t)\in \mathbb Z[t]\quad\text{and}\quad (1-et+t^2)\cdot \To^R_{M,N}(t)\in \mathbb Z[t]\,.$$
\end{Theorem}

When  $l(M\otimes_RN)<\infty$, with $l(-)$ denoting length, one can define modified versions of the series $\Eo_R^{M,N}(t)$ and $\To^R_{M,N}(t)$ as follows: 
$$
\mathcal E^{M,N}_R(t)=\sum_{i=0}^\infty l\left(\Ext^i_R(M,N)\right)t^i \quad\,\text{and}\quad\, 
\mathcal T_{M,N}^R(t)=\sum_{i=0}^\infty l\left(\Tor_i^R(M,N)\right)t^i \,.
$$

Under the assumptions of the Theorem, our proof reveals that $\fm \Ext^i_R(M,N)=0$ and $\fm \Tor_i^R(M,N)=0$
for $i\gg 0$, hence we also have, cf.\! Corollary \ref{end}: 
$$(1-et+t^2)\cdot \mathcal E^{M,N}_R(t)\in \mathbb Z[t]\quad\text{and}\quad (1-et+t^2)\cdot \mathcal T^R_{M,N}(t)\in \mathbb Z[t]\,.$$
When $R$ is a complete intersection, rationality of  $\mathcal E^{M,N}_R(t)$ and $\mathcal T_{M,N}^R(t)$  is known, due to Gulliksen \cite{Gu}. On the other hand, Roos \cite{Ro} gives an example of a (non-Gorenstein) ring $R$ with $\fm^3=0$ and a module $M$ such that $\mathcal E^{M,M}_R(t)$ is rational, while 
$\mathcal T_{M,M}^R(t)$ is transcendental. We refer to \cite{Ro} for the connections of such results with homology of free loop spaces and cyclic homology. 

The rings  considered in this paper (i.e. Gorenstein rings with radical cube zero) are homomorphic images of a hypersurface, via a Golod homomorphism (see \cite[1.4]{AIS}).  As indicated by Roos, it is reasonable to expect that the series $\mathcal E^{M,N}_R(t)$ and $\mathcal T_{M,N}^R(t)$ are rational  for all $M$, $N$ with  $l(M\otimes_RN)<\infty$ whenever $R$ is a homomorphic image of a complete intersection via a Golod homomorphism. Along the same lines, we may also expect that the series $\Eo^{M,N}_R(t)$ and $\To^R_{M,N}(t)$ are rational for such $R$, and any finite $R$-modules $M$, $N$.

An important aspect of our arguments is the use of the notion of {Koszul module}.   The structure of Koszul modules in the case of Gorenstein rings $R$  with $\fm^3=0$  is well understood, and is used heavily in the proofs. The main ingredient in the proof consists of showing that, under the hypotheses of the theorem, the homomorphism  $\Tor_i^R(\fm M,N)\to \Tor_i^R(M, N)$ induced by the inclusion $\fm M\hookrightarrow M$ is zero for $i\gg 0$ whenever the module $M$ is Koszul.  This is the statement of Proposition \ref{ExtMNprop}, proved in Section 2.  The proof of the main theorem  is given  in Section 3.

\section{Preliminaries}

In this section we introduce notation and discuss some background. We introduce the notion of Koszul module, and we give equivalent characterizations in the case of interest. Lemma \ref{m^2} and Lemma \ref{3parts} will become instrumental in Section 2 in setting up an induction argument towards the proof of the main result, while Lemma \ref{RmodIhomoN} provides one of the key ideas in constructing the proof. 

Throughout, $(R,\fm,\fk)$ denotes a local commutative ring with maximal ideal $\fm$ and $\fk=R/\fm$, and $M$, $N$ are finite $R$-modules.  We set 
$$\ov M=M/\fm M\quad\text{and}\quad \nu(M)=\rank_{\fk}(\ov M)\,.$$

\begin{lemma}
\label{m^2}
Assume $\fk$ is algebraically closed and $\nu(\fm)\ge 2$. Let $M$ be a finite $R$-module with $\fm^2M=0$ and such that $\nu(M)\ge \nu(\fm M)$. There exists then $x\in M\smallsetminus\fm M$ such that $\ann(x)\ne \fm^2$. 
\end{lemma}

\begin{proof}
Assume that $\ann(x)=\fm^2$ for every $x\in M\smallsetminus \fm M$. If $a\in R$, we denote by $\ov a$ its image in $\fk=R/\fm$. Set $\nu(\fm M)=n$.  Since $\fm^2M=0$, note that $\fm M$ has a vector space structure over $\fk$ and $\rank_{\fk}(\fm M)=n$. The structure is given by  $\ov a x=ax$ for $x\in \fm M$ and $a\in R$.
 
By hypothesis, we have  $\nu(M)\ge n$. Let $x_1, \dots, x_n$ be part of a minimal generating set of $M$.

\noindent{\it Claim:}  If $\alpha\in \fm\smallsetminus \fm^2$, then $\alpha x_1, \dots, \alpha x_n$ is a  basis of $\fm M$ over $\fk$. 

To prove this claim, assume 
$$
\ov{b_1}(\alpha x_1)+\dots +\ov{b_n}(\alpha x_n)=b_1(\alpha x_1)+\dots +b_n(\alpha x_n)=0 
$$
for some $b_i\in R$.  Set $x=b_1x_1+\dots +b_nx_n$. We have thus $\alpha x=0$, hence $\alpha\in \ann(x)$. If  $x\notin\fm M$, then $\ann(x)=\fm^2$ by assumption, and thus  $\alpha\in \fm^2$, a contradiction. Consequenly $x\in \fm M$, and hence $b_i\in \fm$ and thus $\ov{b_i}=\ov{0}$ for all  $i$. This shows that  $\alpha x_1, \dots ,\alpha x_n$ is linearly independent over $\fk$. Since $\rank_{\fk}(\fm M)=n$, this set is a basis of $\fm M$ and the claim is proved.

Assume now that $\alpha, \beta$ is part of a minimal set of generators of $\fm$. By the above, the sets $\alpha x_1, \dots ,\alpha x_n$ and $\beta x_1, \dots ,\beta x_n$ are both bases of $\fm M$ over $\fk$. We have then relations 
\begin{equation}
\label{beta}
\beta x_j=\sum_{i=1}^n  p_{ij}\alpha x_i\qquad \text{for all $j$ with $1\le j\le n$, }
\end{equation}
where $p_{ij}\in R$ and the change of basis matrix $P=(\ov{ p_{ij}})$ is invertible.    Recall that $\fk$ is algebraically closed, and let $\lambda\in \fk$ be an eigenvalue of $P$. Since $P$ is invertible we have $\lambda\ne 0$. Choose then $\gamma\in R$ so that $\ov \gamma=-\lambda^{-1}$. Since $\lambda=-(\ov\gamma)^{-1}$ is an eigenvalue, we have $\det(P+(\ov\gamma)^{-1}I)=\ov 0$, where $I$ is the $n\times n$ identity matrix,  and it follows that $\det(I+\ov \gamma P)=\ov 0$ and hence the  matrix equation 
$$(I+\ov \gamma P){\bd b} =\ov 0$$ has a nontrivial solution ${\bd b}\in \fk^n$, where ${\bd b}=(\ov{b_1}, \dots, \ov{b_n})^T$ with $b_i\in R$. With this choice of $\gamma$ and $b_i$, we have thus
\begin{equation}
\label{b}
 b_i+\gamma \sum_{j=1}^n{p_{ij}} b_j\in \fm \qquad \text{for all $i$ with $1\le i\le n$\,.}
\end{equation}
The equations \eqref{b} and \eqref{beta} yield: 
\begin{align*}
(\alpha+\beta\gamma) (b_1x_1+\dots+ b_nx_n)&=\sum_{i=1}^n b_i(\alpha x_i)+\gamma \sum_{j=1}^n b_j (\beta x_j)=\\
&=\sum_{i=1}^n(b_i+\gamma \sum_{j=1}^np_{ij}b_j)(\alpha x_i)\in \fm^2M=0\,.
\end{align*}
Set $x=b_1x_1+\dots +b_nx_n$ and note that $x\notin \fm M$, since the vector $\bd b\in \fk^n$ is nontrivial, and thus $b_i\notin \fm$ for at least some $i$. We have thus  $\alpha+\beta\gamma\in \ann(x)$ and,  since $\ann(x)=\fm^2$, it follows that $\alpha+\beta\gamma\in \fm^2$. This is a contradiction, since $\alpha, \beta$ is part of a minimal set of generators for $\fm$. 
\end{proof}

 Let $\varphi\col M\to N$ be a homomorphism. We denote by $\ov \varphi$ the induced map $\ov\varphi\col \ov M\to \ov N$. 
If $A$ is a finite  $R$-module, then for each $i$ we let 
\begin{align*}
&\Tor_i^R(\varphi, A)\colon \Tor_i^R(M, A)\to \Tor_i^R(N, A)\\
&\Ext^i_R(\varphi,M)\colon \Ext^i_R(N,A)\to \Ext^i_R(M,A)
\end{align*}
denote the induced maps. We let 
$$
\iota_M\col \fm M\to M\qquad\text{and}\qquad \pi_M\col M\to  \ov M
$$ 
denote the  inclusion, respectively the canonical projection.

For each $i$ we set $\beta_i^R(M)=\rank_{\fk}\Tor_i^R(M,\fk)$ ; this number is the $i$th {\it Betti number} of $M$ over $R$.

The main ingredient  in proving rationality of the series defined in the introduction consists in  showing that the maps $\Tor_i^R(\iota_M,N)$  become zero for large values of $i$, under certain assumptions on the ring and on the modules. The next lemma is a first step in this direction, and it will be further extended in Section 2. 
 
\begin{lemma}\label{RmodIhomoN}
Let $M$, $N$ be finite $R$-modules with $\nu(M)=1$.  Assume there exists an  integer $i\ge 0$ such that $\Tor_i^R(\iota_{M},\fk)=0$ and $\beta_i^R(M)>\beta_i^R(N)$. 
\begin{enumerate}[\quad\rm(1)]
\item If $\varphi\in \Hom_R(M,N)$, then  $\varphi(M)\subseteq \fm N$.
\item If $\fm^2N=0$ then $\Hom_R(\iota_{M}, N)=0$.
\end{enumerate}
\end{lemma}

\begin{proof}
Assume $\varphi(M)\not\subseteq \fm N$. Since $M$ is cyclic, the induced map ${\ov\varphi} \colon \ov{M} \to \ov N$ is injective. Since it is a homomorphism of vector spaces, it has a splitting, hence the induced maps $\Tor_i^R({\ov \varphi},\fk)\col \Tor_i^R(\ov M,\fk)\to \Tor_i^R(\ov  N,\fk)$ are injective for $i \ge 0$.  The short exact sequence 
$$
0\to \fm M\to M\to \ov M\to 0
$$
induces the top exact row in the commutative diagram below: 
\begin{equation}
\begin{tikzcd}
\Tor_i^R(\fm M,\fk)\arrow{rr}{{\Tor_i^R(\iota_M,\fk)}}& & \Tor_i^R(M,\fk)\arrow{rr}{{\Tor_i^R(\pi_M,\fk)}}\arrow{d}[swap] {\Tor_i^R(\varphi,\fk)}& 
&\Tor_i^R(\ov M,\fk)\arrow{d}{{\Tor_i^R(\ov \varphi, \fk)}}\\
&& \Tor_i^R(N,\fk)\arrow{rr}{\Tor_i^R(\pi_N,\fk)}& & \Tor_i^R(\ov N,\fk)
\end{tikzcd}
\end{equation}
If $\Tor_i^R(\iota_M,\fk)=0$, then $\Tor_i^R(\pi_M,\fk)$ is injective. 
The commutative square gives then that $\Tor_i^R(\varphi,\fk)$ is injective. We conclude $\beta_i^R(M)\le \beta_i^R(N)$, a contradiction. 

We have thus $\varphi(M)\subseteq \fm N$, hence (1) is established.  To prove (2), note that the image of $\varphi$ under the map  $\Hom_R(\iota_{M}, N)$ is the composition $\varphi\iota_{M}\col \fm M\to N$. We have: 
$$
\varphi\iota_M(\fm M)\subseteq \varphi(\fm M)\subseteq \fm \varphi(M)\subseteq \fm^2N\,.
$$
When $\fm^2N=0$ we conclude  $\varphi\iota_M=0$, and thus $\Hom_R(\iota_M,N)=0$.  
\end{proof}

\begin{bfchunk}{Hilbert and Poincar\'e series. }
The \emph{Hilbert series} of $M$ (over $R$) is the formal power series
$$
H_M(t)=\sum_{i=0}^\infty \rank_\fk\left(\frac{\fm^i M}{\fm^{i+1}M}\right)t^i\,.
$$ 
The {\it Poincar\'e series} of $M$ is the formal power series
$$
P_M^R(t)=\sum_{i=0}^\infty \beta_i^R(M)t^i\,.
$$
\end{bfchunk}

The next remark clarifies  the attention we will  give in Section 2 to the vanishing of the maps $\Tor_i^R(\iota_M,N)$: such vanishing allows for computations of the series of interest. 

\begin{remark}
\label{length-comp}
Assume that  $\fm^2M=0$. The short exact sequence
$$
0\to \fm M\xrightarrow{\iota_M} M\xrightarrow{\pi_M} \ov M\to 0
$$
induces for each $i>0$ the following exact sequence: 
$$
0\to L_i\to \Tor_i^R(M,N)\xrightarrow{\Tor_i^R(\pi_M,N)}\Tor_i^R(\ov M,N)\xrightarrow{\Delta_i}\Tor_{i-1}^R(\fm M,N)\to L_{i-1}\to 0
$$
where $L_i$ is the image of the map $\Tor_i^R(\iota_M,N)$. A length count gives: 
$$
l(\Tor_i^R(M,N))=l(\Tor_i^R(\ov M,N))-l(\Tor_{i-1}^R(\fm M,N)) +l(L_i)+l(L_{i-1})
$$
Since both $\ov M$ and $\fm M$ are $\fk$-vector spaces, we have 
\begin{align*}
l(\Tor_i^R(\ov M,N))&=\rank_{\fk}(\Tor_i^R(\fk^{\nu(M)}, N))=\nu(M)\beta_i^R(N);\\
l(\Tor_{i-1}^R(\fm M,N))&=\rank_{\fk}(\Tor_{i-1}^R(\fk^{\nu(\fm M)}, N))=\nu(\fm M)\beta_{i-1}^R(N)\,.
\end{align*}
We have thus
\begin{equation}
\label{length}
l(\Tor_i^R(M,N))\ge \nu(M)\beta_i^R(N)-\nu(\fm M)\beta_{i-1}^R(N)\,.
\end{equation}
Equality holds in \eqref{length} if and only if $L_i=0=L_{i-1}$, and hence  if and only if $\Tor_i^R(\iota_M,N)=0=\Tor_{i-1}^R(\iota_M,N)$.
In particular,  we obtain from here that the following two statements are equivalent when $\fm^2M=0$: 
\begin{enumerate}
\item $\Tor_i^R(\iota_M, N)=0$ for all $i\ge 0$; 
\item $\sum_{i=0}^\infty l(\Tor_i^R(M,N))t^i=H_M(-t)\Po^R_{N}(t)$. 
\end{enumerate}
Also, note that $L_i=0$ implies that $\Tor_i^R(M,N)$ is isomorphic to a submodule of $\Tor_i^R(\ov M,N)$, and hence $\fm \Tor_i^R(M,N)=0$. In particular, condition (1) also implies: 
\begin{enumerate}
\item[(3)]  $\sum_{i=0}^\infty \nu(\Tor_i^R(M,N))t^i=H_M(-t)\Po^R_{N}(t)$. 
\end{enumerate}
\end{remark}

\begin{bfchunk}{Koszul rings and modules.}
\label{defKos}
As defined in \cite{HI}, an $R$-module $M$ is said to be {\it Koszul} if its linearity defect is $0$; we refer to {\it loc.\! cit.\!} for the definition of linearity defect, and we note that $M$ is Koszul if and only if the associated graded module $\gr_{\fm}(M)$ has a linear resolution over $\gr_{\fm}(R)$. As noted in \cite[1.8]{HI}, if $M$ is Koszul, then 
\begin{equation}
\label{KoszulP}
\Po_M^R(t)=\frac{H_M(-t)}{H_R(-t)}
\end{equation}
The ring $R$ is said to be {\it Koszul} if $\fk$ is a Koszul module. 

If $R$ is Koszul and $\fm^2M=0$, then the following are equivalent: 
\begin{enumerate}
\item $M$ is Koszul; 
\item $\Tor_i^R(\iota_M,\fk)=0$ for all $i\ge 0$; 
\item The formula \eqref{KoszulP} holds. 
\end{enumerate}
See \cite[3.1]{AIS} for the equivalence $(1)\iff(2)$ and  Remark \ref{length-comp} for $(2)\iff(3)$. 
\end{bfchunk}

\begin{lemma}
\label{3parts}
Assume there exists a short exact sequence
$$
0\to A\xrightarrow{\varphi} M\xrightarrow{\psi}B\to 0
$$
of finite $R$-modules such that $\ov\varphi\colon \ov A\to \ov M$ is injective. Let $N$ be a finite $R$-module. 
\begin{enumerate}[\quad\rm(1)]
 \item If $\Tor_i^R(\iota_A,N)=0$ for some $i$, then $\Tor_i^R(\varphi, N)$ is injective and $\Tor_{i+1}^R(\psi, N)$ is surjective. 
\item If $\Tor_i^R(\iota_B,N)=\Tor_{i-1}^R(\iota_A,N)=\Tor_i^R(\iota_A,N)=0$ for some $i$, then we also have $\Tor_i^R(\iota_M,N)=0$. 
\item If $R$ is a  Koszul ring, $\fm^2M=0$ and $M$ is Koszul,  then $B$ is Koszul.
\end{enumerate}
\end{lemma}

\begin{proof}
The hypothesis that  $\ov\varphi$ is injective yields a commutative diagram with exact rows and columns: 
\begin{equation}
\label{diag1}
\begin{tikzcd}
&0\arrow{d}{}&0\arrow{d}{}&0\arrow{d}{}&\\
0\arrow{r} &\fm A\arrow{d}{\iota_A}\arrow{r}{} & \fm M\arrow{r}{\psi'}\arrow{d}{\iota_M} & \fm B\arrow{r}{}\arrow{d}{\iota_B}& 0\\
0\arrow{r}& A\arrow{d}{\pi_A}\arrow{r}{\varphi} & M\arrow{r}{\psi}\arrow{d}{\pi_M} & B\arrow{r}\arrow{d}{\pi_B}& 0\\
0\arrow{r}& \ov A\arrow{r}{\ov{\varphi}}\arrow{d}{} & \ov M\arrow{r}{\ov{\psi}}\arrow{d}{}&\ov  B\arrow{r}\arrow{d}{}& 0\\
&0&0&0&
\end{tikzcd}
\end{equation}
Note that the bottom row in this diagram is an exact sequence of vector spaces, hence it is  split, and it remains exact when applying $\Tor_i^R(-,N)$.
In particular, $\Tor_i^R(\ov\varphi, N)$ is injective and $\Tor_i^R(\ov \psi, N)$ is surjective. 
 Diagram \eqref{diag1} induces then the following commutative diagram with exact rows and columns: 
\begin{equation}
\label{diag2}
\begin{tikzcd}[column sep=large]
& \Tor_i^R(\fm A,N)\arrow{d}{\Tor_i^R(\iota_A,N)}\arrow{r}{} & \Tor_i^R(\fm M,N)\arrow{d}{\Tor_i^R(\iota_M,N)}\arrow{r}{\Tor_i^R(\psi',N)} &\Tor_i^R(\fm B,N)\arrow{d}{{\Tor_i^R(\iota_B,N)}}& \\
& \Tor_i^R(A,N) \arrow{d}{\Tor_i^R(\pi_A,N)}\arrow{r}{{\Tor_i^R(\varphi,N)}}& \Tor_i^R(M,N)\arrow{d}{\Tor_i^R(\pi_M,N)}\arrow{r}{\Tor_i^R(\psi,N)}\ &  \Tor_i^R(B,N)\, \arrow{d}{{\Tor_i^R(\pi_B,N)}}&\\
&0\to  \Tor_i^R(\ov A,N)\arrow{r}{\Tor_i^R(\ov{\varphi},N)} & \Tor_i^R(\ov M,N)\arrow{r}{\, \Tor_i^R(\ov{\psi},N)} &\Tor_i^R(\ov B,N)\to 0&
\end{tikzcd}
\end{equation}
We have then:

(1) If  $\Tor_i^R(\iota_A,N)=0$, it follows that $\Tor_i^R(\pi_A,N)$ is injective. Since   $\Tor_i^R(\ov\varphi, N)$ is injective, the bottom left commutative square gives that $\Tor_i^R(\varphi, N)$ is injective as well. The fact that $\Tor_{i+1}^R(\psi, N)$ is surjective follows from the long exact sequence associated in homology with the exact sequence from the statement.

(2) In view of part (1),  the hypothesis that $\Tor_{i-1}^R(\iota_A,N)=\Tor_i^R(\iota_A,N)=0$ shows that $\Tor_i^R(\varphi, N)$ is injective and $\Tor_{i}^R(\psi, N)$ is surjective. The hypothesis also implies that $\Tor_i^R(\pi_A,N)$ and $\Tor_i^R(\pi_B,N)$ are injective. A snake lemma argument using the bottom two rows of \eqref{diag2} gives that $\Tor_i^R(\pi_M,N)$ is injective as well, and thus $\Tor_i^R(\iota_M,N)=0$. 

(3) The additional hypothesis that $\fm^2M=0$  gives that the top row in \eqref{diag1} is an exact sequence of vector spaces. Consequently, it  is split, and in particular  $\Tor_i^R(\psi',\fk)$ is surjective for all $i\ge 0$.  Since $M$ is Koszul, we have $\Tor_i^R(\iota_M,\fk)=0$ for all $i\ge 0$. The upper right commutative square in \eqref{diag2} with $N=\fk$ yields then $\Tor_i^R(\iota_B,\fk)=0$ for all $i\ge 0$, and hence $B$ is Koszul, in view of \ref{defKos}. 
\end{proof}

\section{ Koszul modules over short Gorenstein rings}

In this section we focus our attention on Gorenstein local rings with $\fm^3=0$. We present first some needed background material.  The bulk of the section is taken up with the proof of Proposition \ref{ExtMNprop} and its supporting lemmas.

Throughout this section, $(R,\fm,\fk)$ denotes a Gorenstein local ring with $\fm^3=0$ and $\fm^2\ne 0$. We set 
$
e=\nu(\fm)
$ 
and we assume that $e\ge 2$. 

Let $M$, $N$ be finite $R$-modules. For any $N$ set  $N^*=\Hom_R(N,R)$.  If $\varphi\col M\to N$ is a homomorphism, then $\varphi^*\col N^*\to M^*$ denotes the induced map. 

\begin{bfchunk}{Syzygies.}
Let $M$ be a finitely generated $R$-module and let
\begin{equation}
\label{freeM}
\dots \to F_i\xrightarrow{\dd_i} F_{i-1}\to \dots\to  F_1\xrightarrow{\dd_1} F_0\to 0
\end{equation}
be a minimal free resolution of $M$ over $R$. Note that the  Betti numbers of $M$ can be read off this resolution, namely  $\beta_i^R(M)=\rank_R(F_i)$ for all $i\ge 0$.
We set $M_0=M$ and for each $i> 0$ we set 
$$
M_i=\Ima(\partial_i)\,.
$$
The module $M_i$ is called the $i$th {\it syzygy} of $M$. Since $\fm^3=0$, the minimality of the resolution shows that $\fm^2M_i=0$ for all $i>0$. 
 Now let 
\begin{equation}
\dots \to G_i\xrightarrow{d_i} G_{i-1}\to \dots \to G_1\xrightarrow{d_1} G_0\to 0
\end{equation}
be a minimal free resolution of $M^*$. Since $R$ is Gorenstein and Artinian,  the dual of this resolution is also exact and gives a minimal injective resolution of $M^{**}$:
\begin{equation}
\label{injM}
0\to F_{-1}\xrightarrow{\dd_{-1}} F_{-2}\to \dots \to F_{-i}\xrightarrow{\dd_{-i}} F_{-i-1}\to \dots
\end{equation}
with $F_{-i}=G^*_{i-1}$ and $\dd_{-i}=d_{i}^*$. 

Since $M\cong M^{**}$, note that the resolutions in \eqref{freeM} and \eqref{injM} can be ``glued" together through a map $\dd_0$,  yielding a {\it complete resolution} of $M$: 
\begin{equation}
\dots\to F_i\xrightarrow{\dd_i} F_{i-1}\to \dots \to F_1\xrightarrow{\dd_1} F_0\xrightarrow{\dd_0}F_{-1}\xrightarrow{\dd_{-1}} F_{-2}\to \dots \to F_{-i}\xrightarrow{\dd_{-i}} F_{-i-1}\to \dots
\end{equation}
This complex is acylic, that is, its homology is zero in each degree.  If $i>0$ we set 
$$
M_{-i}=\Ima(\dd_{-i})
$$
If $\fm^2M=0$, then $\dd_0(F_0)\subseteq \fm F_{-1}$ and the complete resolution is minimal. Consequently, if $M$ and $N$ are two $R$-modules with $\fm^2M=0=\fm^2N$, the minimal complete resolution shows that 
$$
M_{-i}\cong N \iff M\cong N_i\quad\text {for all $i$.}
$$

In a similar manner, we define negative Betti numbers, when $\fm^2M=0$, by setting $\beta_{-i}^R(M)=\rank_R(F_{-i})$ for all $i>0$. In particular, we have: 
\begin{gather*}
\beta_{-i}^R(M)=\rank_R(F_{-i})=\rank_R(G_{i-1})=\beta_{i-1}^R(M^*) \\
M_{-i}=\Ima(\dd_{-i})=\Ima(d_i^*)\cong (\Ima d_i)^*=(M^*)_i
\end{gather*} 
Furthermore, since $\fk^*\cong \fk$, we have $\beta_{-i}^R(\fk)=\beta_{i-1}^R(\fk)$ and $\fk_{-i}\cong \fk_i$.
\end{bfchunk}

\begin{bfchunk}{Koszul modules over short Gorenstein rings.}\label{KoszulGor}
With $R$ as above, the following statements are equivalent (see \cite[4.6]{AIS}):
\begin{enumerate}
\item $M$ is Koszul; 
\item The syzgygy $M_i$ does not split off a copy of $\fk$ for any $i>0$ (equivalently $M$ is {\it exceptional}, using the terminology of Lescot \cite{Le});
\item $M$ has no direct summand isomorphic to $\fk_{-i}$,  for all $i>0$. 
\end{enumerate}
In particular,  it follows, as noted in \cite[4.6]{AIS}, that  an indecomposable  module $M$ over the short Gorenstein ring $R$ is Koszul if and and only if $M$ is not isomorphic to $\fk_{-i}$ for any $i>0$. Also, note that if $M$ is Koszul, then $M_i$ is Koszul for all $i>0$. 

L\"ofwall \cite{Lof} shows that a Gorenstein ring with $\fm^3=0$ and $e\ge 2$ satisfies 
\begin{equation}
\label{Pok} \Po_\fk^R(t)=\frac{1}{1-et+t^2}.
\end{equation}
This result is recovered and used by Sj\"odin \cite{Sj} to show that for every finitely generated $R$-module one has: 
\begin{equation}
(1-et+t^2)\cdot \Po^R_M(t)\in \mathbb Z[t]\,.
\end{equation}
The results of \cite{Sj} are recovered in \cite{AIS}, where it is also noted that any such $R$ is Koszul.  

Note that formula \eqref{Pok} shows that the Betti numbers $b_i=\beta_i^R(\fk)$ satisfy the relations $b_0=1$, $b_1=e$  and $b_{i+1}=eb_i-b_{i-1}$ for all $i\ge 1$. Since we assumed $e\ge 2$, it follows inductively that the sequence $\{\beta_i^R(\fk)\}_{i\ge 0}$ is strictly increasing. 
\end{bfchunk}

\begin{remark}
\label{Lescot}
It is known that $\fm^2 M=0$ when $M$ is indecomposable and not free; see for instance the proof of \cite[4.6]{AIS}. 

Also, if $\fm^2M=0$ and $M_1$ does not split off a copy of $\fk$, then the following formulas hold, cf.\! \cite[3.3]{Le}: 
\begin{align}
\label{f1}
\nu(M_1)&=\nu(M)e-\nu(\fm M);\\
\label{f2}
\nu(\fm M_1)&=\nu(M)\,.
\end{align}
\end{remark}

\begin{lemma}\label{Ieqm2}
Let $I$ be an ideal of $R$. Then $R/I$ is not Koszul if and only if $I=\fm^2$. 
\end{lemma}

\begin{proof}
Since $R$ is Gorenstein with socle $\fm^2$, note that $\fm^2\cong \fk$. 
If $I=\fm^2$ it follows that $\fk\cong (R/I)_1$, hence  $R/I$ is not Koszul. 

Now assume $R/I$ is not Koszul.  Then $R/I \cong \fk_{-i}$ for some $i>0$, hence 
$$
\beta_{i-1}^R(\fk)=\beta_{-i}^R(\fk)=\beta_0^R(R/I)=1\,.
$$
Since the Betti numbers of $\fk$ are strictly increasing, the equality $\beta_{i-1}^R(\fk)=1$ implies $i=1$. We have thus $R/I\cong \fk_{-1}$. Since $\fk_{-1} \cong R/\fm^2$, we conclude $I=\fm^2$. 
\end{proof}

\begin{lemma}\label{bettiRmodI}
If $I$ is a proper ideal of $R$ and $e>2$, then the sequence $\{\beta_i^R(R/I) \} _{i\ge 1}$ is strictly increasing and $\beta_i^R(R/I) \ge i$ for all $i \ge 0$.
\end{lemma}

\begin{proof}
If $I=\fm^2$, then $\beta_i^R(R/\fm^2)=\beta_{i-1}^R(\fm^2)=\beta_{i-1}^R(\fk)$ for all $i\ge 1$, and the conclusion follows from the fact that the sequence  $\{\beta_i^R(\fk)\}_{i\ge 0}$ is strictly increasing. 

Assume now that $I\ne \fm^2$. Since $I \subseteq \fm$, we have $\fm I \subseteq \fm^2$.  Since $I\ne 0$ and $R$ is Gorenstein with socle $\fm^2$,  it follows that $\fm I=\fm^2$  and hence $\fm^2(R/I)=0$. Set $a=\rank_\fk(\fm/I)$. The assumption that $I\ne \fm^2$ gives $a<e$. The Hilbert series of $R/I$ is $H_{R/I}(t)=1 + at$. 
Since $R/I$ is Koszul by Lemma \ref{Ieqm2}, we have: 
\begin{equation}P_{R/I}^R(t)=\frac{1-at}{1-et+t^2}\,.
\end{equation}
Set  $b_i=\beta_i^R(R/I)$ for $i \ge 0$. We have then: 
\begin{equation}1-at=\left( b_0 + b_1t + b_2t^2 + b_3t^3 + \ldots \right) \left( 1-et+t^2 \right).
\end{equation}
From this equation we derive the following information: $b_0=1$, $b_1=e-a$, and $b_{i+2}=eb_{i+1}-b_i$ for $i \ge 0$. Note that $b_1-b_0\ge 0$ because $a<e$. Let $n\ge 1$ and assume $b_{n}-b_{n-1}\ge n-1$. Since $e > 2$  we have 
$$
b_{n+1}-b_{n}=(eb_{n}-b_{n-1})-b_{n}=b_{n}(e-1)-b_{n-1}>b_{n}-b_{n-1} \ge n-1
$$
hence $b_{n+1}-b_{n}\ge n$. 
This inductive argument gives that $b_{i+1}-b_i\ge i$ for all $i\ge 0$. In particular,  $b_i\ge i$ for all $i\ge 0$ and the sequence $\{b_i\}_{i\ge 1}$ is strictly increasing. 
\end{proof}

\begin{lemma}
\label{ExtMN}
Assume $e>2$. If $M$ is Koszul with $\nu(M)=1$, then $\Ext_R^i(\iota_M,N)=0$ for all $i$ with $i > \nu(N)$. Equivalently, $\Tor_i^R(\iota_M,N)=0$ for all $i$ with $i>\nu(N^*)$. 
\end{lemma}

\begin{proof}
We may assume $M$, $N$ are indecomposable and not free. In particular, it follows that $\fm^2M=0=\fm^2N$.
 Let $i$ be such that $i>\nu(N)$ and set $L=N_{-i}$. Note that $\fm^2L=0$ and 
$$\nu(N)=\beta_0(N)=\beta_i(N_{-i})=\beta_i(L)\,.$$ 

Since $M$ is cyclic, we have $M\cong R/I$ for a proper ideal $I$. Lemma \ref{bettiRmodI} gives that $\beta_i^R(M)\ge i$, and hence $\beta_i^R(M)>\beta_i^R(L)$, since $i>\nu(N)$.  Since $M$ is Koszul, we have $\Tor_i^R(\iota_M, \fk)=0$ for all $i$, and Lemma \ref{RmodIhomoN} gives that $\Hom_R(\iota_M,L)=0$. 

For each $n\ge 0$ extract from a minimal complete resolution of $N$ the short exact sequence
\begin{equation}0 \lra N_{-i+n+1} \lra R^c \lra N_{-i+n} \lra 0
\end{equation}
with $c=\beta_{-i+n}(N)$, and consider the induced commutative diagram with exact rows: 
\begin{equation}
\label{delta}
\begin{tikzcd}
\Ext_R^n(M,N_{-i+n})\arrow{d}{\Ext_R^n(\iota_M,N_{-i+n})}\arrow{r}{\Delta_{n+1}} & \Ext_R^{n+1}(M,N_{-i+n+1})\arrow{d}{\Ext_R^{n+1}(\iota_M,N_{-i+n+1})}\arrow{r} & \Ext_R^{n+1}(M,R^c)\arrow{d} \\
\Ext_R^n(\fm M,N_{-i+n})\arrow{r} & \Ext_R^{n+1}(\fm M,N_{-i+n+1})\arrow{r} & \Ext_R^{n+1}(\fm M,R^c)
\end{tikzcd}
\end{equation}
We prove by induction on $n$ that $\Ext_R^n(\iota_M,N_{-i+n})=0$ for all $n\ge 0$. This holds for $n=0$, because we know $\Hom_R(\iota_M,L)=0$.

 Assume now that  $n\ge 0$ and $\Ext_R^n(\iota_M,N_{-i+n})=0$. The connecting homomorphism $\Delta_{n+1}$ in \eqref{delta} is surjective because we have $\Ext_R^{n+1}(M,R^c)=0$, since $R$ is Gorenstein artinian. (It is an isomorphism when $n\ge 1$.)
The commutative square on the left gives that $\Ext_R^{n+1}(\iota_M,N_{-i+n+1})=0$. 

We have thus  $\Ext_R^n(\iota_M,N_{-i+n})=0$ for all $n\ge 0$. Taking $n=i$ and noting that $N_0=N$, we obtain the desired conclusion that  $\Ext_R^i(\iota_M,N)=0$ for all $i>\nu(N)$. In particular, we have $\Ext_R^i(\iota_M,N^*)=0$ for all $i>\nu(N^*)$. Finally, note that $\Ext_R^i(\iota_M,N^*)=0$ if and only if  $\Tor^R_i(\iota_M, N)=0$, in view of the canonical isomorphisms given by duality. 
\end{proof}

\begin{lemma}
\label{iff}
Assume $\fm^2M=0$. 
If  $M_1$ does not split off a copy of $\fk$, then 
$\Tor_i^R(\iota_M,N)=0$ for $i\gg 0$ iff $\Tor_i^R(\iota_{M_{1}},N)=0$ for $i\gg 0$.
\end{lemma}

\begin{proof}
By \eqref{length}, we have inequalities
\begin{gather}
\label{1}
l(\Tor_{i+1}^R(M,N))\ge \nu(M)\beta_{i+1}(N)-\nu(\fm M)\beta_i(N)\,;\\
\label{2}
l(\Tor_i^R(M_1,N)) \ge \nu(M_1)\beta_{i}(N)-\nu(\fm M_1)\beta_{i-1}(N)\,.
\end{gather}
We have $\Tor_i^R(\iota_M,N)=0$ for $i\gg 0$ if and only if \eqref{1} is an equality for $i\gg 0$, and $\Tor_i^R(\iota_{M_1},N)=0$ for all $i\gg 0$ if and only if \eqref{2} is an equality for $i\gg 0$. Since $\Tor_i^R(M_1,N)\cong \Tor_{i+1}^R(M,N)$,  it suffices to show that 
\begin{equation}
\label{nu}
 \nu(M)\beta_{i+1}(N)-\nu(\fm M)\beta_i(N)= \nu(M_1)\beta_{i}(N)-\nu(\fm M_1)\beta_{i-1}(N)
\end{equation}
for $i\gg 0$. 
Since  the Poincar\'e series of $N$ is rational with denominator $1-et+t^2$, we have 
\begin{equation}
\label{beta1}
\beta_{i+1}(N)=e\beta_i(N)-\beta_{i-1}(N)\qquad \text{for $i\gg 0$}\,.
\end{equation}
Let $i$ be large enough so that \eqref{beta1} holds. 
Using first \eqref{beta1} and then \eqref{f1} and   \eqref{f2}, we establish \eqref{nu} as follows: 
\begin{align*}
 \nu(M)\beta_{i+1}(N)-\nu(\fm M)\beta_i(N)&=\nu(M)(e\beta_i(N)-\beta_{i-1}(N))-\nu(\fm M)\beta_i(N)=\\
&=(\nu(M)e-\nu(\fm M))\beta_i(N)-\nu(M)\beta_{i-1}(N)\\
&=\nu(M_1)\beta_{i}(N)-\nu(\fm M_1)\beta_{i-1}(N)\,.
\end{align*}
As noted above, this finishes the proof. 
\end{proof}

We are now ready to eliminate the assumption that $\nu(M)=1$ in Lemma \ref{ExtMN}. 

\begin{proposition}
\label{ExtMNprop}
If $e>2$ and $M$ is Koszul, then $\Tor_i^R(\iota_M,N)=0$ for $i\gg 0$. 
\end{proposition}

\begin{remark}
If $e=2$, then the conclusion of the proposition may not hold.  Indeed, if $R=k[x,y]/(x^2,y^2)$ and $N=R/(x)$, then a minimal free resolution of $N$ over $R$ is 
$$
\dots \to R\xrightarrow{x}R\to \dots \to R\xrightarrow{x}R\xrightarrow{x}R
$$
 hence $\Tor_i^R(M, N)\cong M$ for any $M$ with $xM=0$. When $M=R/(x)$ as well, the map $\Tor_i^R(\iota_M,N)$ can thus be indentified with the inclusion $\fm M\hookrightarrow M$. 
\end{remark}

\begin{proof}[Proof of Proposition {\rm \ref{ExtMNprop}}]
Let $(R',\fm',\fk')$ be a local ring with $\fk'$ algebraically closed, where $R\to R'$ is an inflation in the sense of \cite[App., Th\'eor\`eme 1, Corollaire]{Bo}, that is: $R'$ is flat over $R$ and $\fm'=R'\fm$. For each finite $R$-module we set $M'=M\otimes_RR'$. As noted in \cite[1.8]{AIS}, $M$ is Koszul if and only if $M'$ is Koszul over $R'$. 
Also, note that we can make the identifications $(\ov M)'=M'/\fm'M'$ and $(\fm M)'=\fm'M'$. The maps $\Tor_i^R(\pi_M,N)$ and $\Tor_i^{R'}(\pi_{M'}, N')$ are simultaneously injective, since $R\to R'$ is faithfully flat, hence  $\Tor_i^R(\iota_M,N)=0$ if and only if $\Tor_i^{R'}(\iota_{M'},N')=0$.  We may assume thus that $\fk$ is algebraically closed. 

We may also assume that $M$ is indecomposable and non-free, and this implies $\fm^2M=0$ as in Remark \ref{Lescot}. 

We prove by induction on $n$ the following statement: \begin{enumerate}
\item[]If $M$ is a Koszul $R$-module such that  $\fm^2M=0$ and $\nu(\fm M)=n$, then $\Tor_i^R(\iota_M,N)=0$ for $i\gg 0$. 
\end{enumerate}
The statement is trivially true when $n=0$, since $\fm M=0$ in this case.  Let $n\ge 1$ and assume that $\Tor_i^R(\iota_M,N)=0$ for $i\gg 0$ for all Koszul modules $M$ with $\fm^2M=0$ and $\nu(\fm M)\le n-1$. 

Let $M$ be a  Koszul $R$-module with $\fm^2M=0$ and  $\nu(\fm M)=n$. We will show that  $\Tor_i^R(\iota_M,N)=0$ for $i\gg 0$. It suffices to establish the conclusion when $M$ is indecomposable, so we will assume this. 

\begin{Case1} Assume $\nu(M)\le n-1$. In this case we have  $\nu(\fm M_{1})=\nu(M)\le n-1$ by Remark \ref{Lescot}.  Since $M$ is Koszul, note that $M_1$ is Koszul and $M_1$ does not split off a copy of $\fk$.  The induction hypothesis, applied to $M_1$,  shows that $\Tor_i^R(\iota_{M_1}
,N)=0$ for $i\gg 0$ and then Lemma \ref{iff} gives $\Tor_i^R(\iota_M,N)=0$ for $i\gg 0$. 
\end{Case1}

\begin{Case2} Assume   $\nu(M)\ge n$.  By Lemma \ref{m^2}, there exists $x\in M\smallsetminus \fm M$ such that  $\ann(x)\ne \fm^2$.  Set  $A=Rx$ and $B=M/A$. Notice the map $\ov A\to \ov M$ induced by the inclusion $A\hookrightarrow M$ is injective, because $x\notin \fm M$.  
 If  $\ann(x)= \fm$ then $A\cong \fk$ and this implies that $M$ splits off a copy of $\fk$, hence $M\cong \fk$, since $M$ is assumed indecomposable. In this case, the statement holds trivially, since $\fm M=0$. We may assume thus $\ann(x)\ne\fm$ as well. 

Since $\ann(x)\ne \fm^2$, Lemma \ref{Ieqm2} shows that $A$ is Koszul. It follows that $\Tor_i^R(\iota_A,N)=0$ for $i\gg 0$ by Lemma \ref{ExtMN}.  Since $\fm^2M=0$, we also have $\fm^2A=0=\fm^2B$ and the top exact row in the commutative diagram \eqref{diag1} is an exact sequence of vector spaces
$$
0\to \fm A\to \fm M\to \fm B\to 0,
$$
which gives
$$
n=\nu(\fm M)=\nu(\fm A)+\nu(\fm B).
$$
Since $\ann(x)\ne \fm$, we have $\nu(\fm A)\ne 0$, and hence  $\nu(\fm B)\le n-1$. Note that $B$ is Koszul by Lemma \ref{3parts}(c) and the  induction hypothesis gives $\Tor^R_i(\iota_B,N)=0$ for $i\gg 0$. Lemma \ref{3parts}(b) gives then $\Tor_i^R(\iota_M,N)=0$ for $i\gg 0$. 
\end{Case2}
The induction argument  is finished, establishing thus the conclusion. 
\end{proof}

\section{Proof of the main theorem}

In this section we prove the main theorem stated in the Introduction. 

\begin{theorem}
\label{primary}
Let $(R,\fm,\fk)$ be a local  Gorenstein ring with  $\fm^3=0\ne \fm^2$ and set $e=\nu(\fm)$. If $e>2$ and  $M$, $N$ are finitely generated $R$-modules, then the following hold: 
\begin{enumerate}[\quad\rm(1)]
\item $\fm\Tor_i^R(M,N)=0$ for  $i\gg 0$; 
\item $\fm\Ext^i_R(M,N)=0$ for $i\gg 0$;
\item $(1-et+t^2)\cdot\To^R_{M,N}(t)\in \mathbb Z[t]$; 
\item $(1-et+t^2)\cdot \Eo_R^{M,N}(t)\in \mathbb Z[t]$. 
\end{enumerate}
\end{theorem}

\begin{proof}
The statements (2) and (4) follow from the statements (1), respectively (3) by duality. We prove below (1) and (3). 

We may assume that both $M$ and $N$ are indecomposable and not free. In particular $\fm^2M=0=\fm^2N$. Let $j\ge 0$. Since $\Tor_{i+j}^R(M,N)\cong \Tor_{i}^R(M_j,N)$ for all $i\ge 1$, the statement (1) holds if and only if $\fm\Tor_i^R(M_j,N)=0$ for $i\gg 0$ and the statement (3) holds if and only if $(1-et+t^2)\cdot\To^R_{M_j,N}(t)\in \mathbb Z[t]$. 

Assume first that $M$ is not Koszul, hence $\fk \cong M_j$ for some $j\ge 1$ (see Section \ref{KoszulGor}). In view of the above observation, it suffices to prove the statement for $M=\fk$ and in this case (1) is clear, and (3) follows from the fact that $\To_{\fk,N}^R(t)=\Po^R_N(t)$ is rational with denominator $1-et+t^2$, as proved by Sj\"odin \cite{Sj}. 

Assume now that $M$ is a Koszul module.  Proposition \ref{ExtMNprop} gives that  there exists an integer $s$ such that $\Tor_i^R(\iota_M,N)=0$ for $i\ge s $.  By Remark \ref{length-comp}, we have that $\fm \Tor_i^R(M,N)=0$ for all $i\ge s$, proving (1), and 
$$
\nu(\Tor^R_{i}(M,N))=l(\Tor_i^R(M,N))=\nu(M)\beta_{i}^R(N)-\nu(\fm M)\beta_{i-1}^R(N)\quad\text{ for all $i>s$}\,.
$$
We have thus 
$$
\To^R_{M,N}(t) =\sum_{i=0}^s\nu(\Tor_i^R(M,N))t^i+\nu(M)\sum_{i\ge s+1}\beta_{i}^R(N)t^i-\nu(\fm M)t\sum_{i\ge s+1}\beta_{i-1}^R(N)t^{i-1}\\
$$
It follows from here that $\To^R_{M,N}(t)-H_M(-t)\Po^R_N(t)\in\mathbb Z[t]$. The conclusion of (3)  follows, using again the fact that $\Po^R_N(t)$ is rational with denominator $1-et+t^2$. 
 \end{proof}

When  $l(M\otimes_RN)<\infty$, we define a modified version of the series $\Eo_R^{M,N}(t)$ and $\To^R_{M,N}(t)$ as follows: 
\begin{gather*}
\mathcal E^{M,N}_R(t)=\sum_{i=0}^\infty l\left(\Ext^i_R(M,N)\right)t^i \in\mathbb Z[[t]]\, \\
\mathcal T_{M,N}^R(t)=\sum_{i=0}^\infty l\left(\Tor_i^R(M,N)\right)t^i \in\mathbb Z[[t]]\,.
\end{gather*}
Under the assumptions of Theorem \ref{primary},  parts (1) and (2)  of its statement give that  
$$\nu(\Ext^i_R(M,N))=l(\Ext^i_R(M,N))\quad\text{and}\quad \nu(\Tor_i^R(M,N))=l(\Tor_i^R(M,N))
$$
for $i\gg 0$, hence we have the following Corollary. 

\begin{corollary}
\label{end}
Under the hypotheses of Theorem {\rm \ref{primary}}, the following hold: 
\begin{enumerate}[\quad\rm(1)]
\item $(1-et+t^2)\cdot\mathcal E^{M,N}_R(t)\in \mathbb Z[t]$;
\item $(1-et+t^2)\cdot\mathcal T_{M,N}^R(t)\in \mathbb Z[t]$. \qed
\end{enumerate}
\end{corollary}

\begin{remark}
Several classes of local rings, including the one discussed in this paper, are known to satisfy the property that the Poincar\'e series of all finite modules are rational, sharing a common denominator;  see \cite{RossiSega} for a large class of Gorenstein artinian rings. In all known cases, such rings are  homomorphic images of a complete intersection via a Golod homomorphism. As mentioned also in \cite{Ro}, it seems reasonable to expect that similar rationality results for the series $\mathcal T_{M,N}^R(t)$, $ E^{M,N}_R(t)$, $\To^R_{M,N}(t)$ and $\Eo^R_{M,N}(t)$ hold for other  such classes. 
\end{remark}

\end{document}